\documentclass[12pt]{amsart}
\usepackage{amssymb,amsmath,amsthm}
\usepackage{tikz}
\usepackage{cancel}
\usepackage{mathtools,float}
\usepackage[T1]{fontenc}

\usetikzlibrary{arrows,shapes,automata,backgrounds,decorations,petri,positioning}

\theoremstyle{plain}
\newtheorem{thm}{Theorem}[section]
\newtheorem{lem}[thm]{Lemma}

\newtheorem{cor}[thm]{Corollary}

\newtheorem{prop}[thm]{Proposition}
\theoremstyle{definition}
\newtheorem{defn}[thm]{Definition}

\newtheorem{ex}[thm]{Example}

\newcommand{\mf}[1]{\mbox{$\mathfrak #1$}}

\newcommand{\av}{\textup{Av}}
\newcommand{\cont}{\textup{Cont}}
\newcommand{\supmesh}{\textup{\textsf{sup-mesh}}}

\title{Mesh patterns with superfluous mesh}

\author{Bridget Eileen Tenner}
\address{Department of Mathematical Sciences, DePaul University, Chicago, IL 60614}
\email{bridget@math.depaul.edu}

\thanks{Research partially supported by a DePaul University Competitive Research Leave Grant.} 


\begin{document}

\begin{abstract}
Mesh patterns are a generalization of classical permutation patterns that encompass classical, bivincular, Bruhat-restricted patterns, and some barred patterns.  In this paper, we describe all mesh patterns whose avoidance is coincident with classical avoidance, in a sense declaring that the additional data of a mesh was unnecessary for these patterns.  We also describe the permutations having the fewest superfluous meshes, and the permutations having the most, enumerating the superfluous meshes in each case.\\

\end{abstract}

\maketitle

\section{Introduction}

Mesh patterns were introduced by Br\"and\'en and Claesson in \cite{branden claesson} to generalize many existing varieties of permutation patterns, including classical patterns (brought to prominence in \cite{simion schmidt}), vincular patterns (see \cite{babson steingrimsson, steingrimsson}), bivincular patterns (introduced in \cite{bousquet-melou claesson dukes kitaev}), Bruhat-restricted patterns (\cite{woo yong}), and certain cases of barred patterns (introduced in \cite{west-thesis}).  There has also subsequently been an additional generalization by \'Ulfarsson, to marked mesh patterns \cite{ulfarsson}.

A mesh pattern (defined precisely in Section~\ref{section:flavors}) involves both a classical permutation and a region in the plane known as the ``mesh.''  In certain cases, the information contained in this mesh turns out to be unnecessary.  For example, the mesh pattern
$$
\left(123,\{(1,1)\}\right) = \begin{minipage}{.86in}
\begin{tikzpicture}[scale=.5]
\fill[black!20] (1,1) rectangle (2,2);
\foreach \x in {1,2,3} {\draw (\x,0) -- (\x,4); \draw (0,\x) -- (4,\x); \fill[black] (\x,\x) circle (5pt);}
\end{tikzpicture}\end{minipage}
$$
is avoided (respectively, contained) by exactly the same permutations as those avoiding (respectively, containing) the classical pattern $123$.

In this paper we characterize all mesh patterns in which the mesh is superfluous in this way, answering a question of Kitaev.  This result, Theorem~\ref{thm:main}, states that a mesh pattern has this property if and only if it has no configurations of the form depicted in Figure~\ref{fig:enclosed}.

The proof of the main result is given in Section~\ref{section:proof}, and the paper concludes with some related enumerative results for extremal cases of superfluous meshes.

\section{Classical and mesh patterns}\label{section:flavors}

We will use the word \emph{permutation} to refer to an automorphism of the set $[k] = \{1, \ldots, k\}$ for some positive integer $k$, and will say that two sequences of real numbers are \emph{order isomorphic} if they are in the same relative order, denoted ``$\approx$.''  We let $\mf{S}_k$ denote the set of all permutations of $[k]$.

\begin{ex}
$\sqrt{5} \ -\!1 \ 0 \ \approx \ 3\ 1\ 2$.
\end{ex}

A permutation $\pi \in \mf{S}_k$ can be written in one-line notation as the word $\pi(1) \cdots \pi(k)$.  Classical pattern containment and avoidance are defined in terms of this notation.

\begin{defn}
Given permutations $\pi \in \mf{S}_k$ and $\sigma \in \mf{S}_n$, where $k \le n$, we say that $\sigma$ \emph{contains} a $\pi$-pattern if there exist indices $i_1 < \cdots < i_k$ so that
$$\sigma(i_1) \sigma(i_2) \cdots \sigma(i_k) \approx \pi.$$
If there are no such indices, then $\sigma$ \emph{avoids} the pattern $\pi$.
\end{defn}

\begin{ex}\
\begin{itemize}
\item The permutation $42135$ contains the pattern $213$ in five ways:
$$425 \approx 415 \approx 435 \approx 213 \approx 215.$$
\item The permutation $42315$ avoids the pattern $132$.
\end{itemize}
\end{ex}

We can also represent a permutation $\pi \in \mf{S}_k$ graphically, as
$$G(\pi) = \{(i,\pi(i)) : 1 \le i \le k\} \subseteq [1,k] \times [1,k].$$
This will be useful when discussing mesh patterns below.

\begin{ex}\label{ex:graph of 42135}
$$G(42135) = 
\begin{minipage}{1.25in}\begin{tikzpicture}[scale=.5]
\foreach \x in {1,2,3,4,5} {\draw (0,\x) -- (6,\x); \draw (\x,0) -- (\x,6);}
\foreach \x in {(1,4),(2,2),(3,1),(4,3),(5,5)} {\fill[black] \x circle (5pt);}
\end{tikzpicture}
\end{minipage}.$$
\end{ex}

To say that $\sigma$ contains a $\pi$-pattern means that $G(\sigma)$ contains at least one copy of the graph $G(\pi)$.

\begin{ex}\
\begin{itemize}
\item The graph $G(42135)$ depicted in Example~\ref{ex:graph of 42135} contains
$$
G(213) = 
\begin{minipage}{.85in}\begin{tikzpicture}[scale=.5]
\foreach \x in {1,2,3} {\draw (0,\x) -- (4,\x); \draw (\x,0) -- (\x,4);}
\foreach \x in {(1,2),(2,1),(3,3)} {\fill[black] \x circle (5pt);}
\end{tikzpicture}
\end{minipage}
$$
in five ways.  The copy of $G(213)$ that corresponds to the occurrence $425$ is
$$
\begin{minipage}{1.25in}
\begin{tikzpicture}[scale=.5]
\foreach \x in {1,2,3,4,5} {\draw (0,\x) -- (6,\x); \draw (\x,0) -- (\x,6);}
\foreach \x in {(1,4),(2,2),(3,1),(4,3),(5,5)} {\fill[black] \x circle (5pt);}
\foreach \x in {(1,4),(2,2),(5,5)} {\draw \x circle (10pt);}
\end{tikzpicture}\end{minipage}.$$
\item The graph depicted in Example~\ref{ex:graph of 42135} does not contain any copies of
$$
G(132) = 
\begin{minipage}{.85in}\begin{tikzpicture}[scale=.5]
\foreach \x in {1,2,3} {\draw (0,\x) -- (4,\x); \draw (\x,0) -- (\x,4);}
\foreach \x in {(1,1),(2,3),(3,2)} {\fill[black] \x circle (5pt);}
\end{tikzpicture}
\end{minipage}.
$$
\end{itemize}
\end{ex}

In \cite{branden claesson}, Br\"and\'en and Claesson introduced a new type of permutation patterns, called ``mesh patterns.''  These include, as special cases, all classical, (bi)vincular, and Bruhat-restricted patterns, as well as certain barred patterns.

\begin{defn}
A \emph{mesh pattern} is an ordered pair $(\pi,R)$, where $\pi \in \mf{S}_k$ is a permutation, and $R$ is a subset of the $(k+1)^2$ unit squares in $[0,k+1]\times[0,k+1]$, indexed by their lower-left corners.  The set $R$ will be called the \emph{mesh} of the mesh pattern.  Mesh patterns are depicted by drawing $G(\pi)$ and shading all squares in the mesh $R$.
\end{defn}

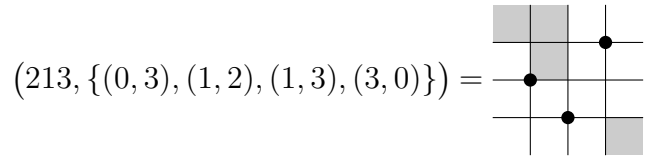
\begin{ex}\label{ex:mesh 213}
$$\big(213,\{(0,3),(1,2),(1,3),(3,0)\}\big) = 
\begin{minipage}{.85in}\begin{tikzpicture}[scale=.5]
\foreach \x in {(0,3),(1,2),(1,3),(3,0)} {\fill[black!20] \x rectangle ++ (1,1);}
\foreach \x in {1,2,3} {\draw (0,\x) -- (4,\x); \draw (\x,0) -- (\x,4);}
\foreach \x in {(1,2),(2,1),(3,3)} {\fill[black] \x circle (5pt);}
\end{tikzpicture}
\end{minipage}.
$$
\end{ex}

\begin{defn}
A permutation $\sigma$ \emph{contains} a mesh pattern $(\pi,R)$ if $G(\sigma)$ contains a copy of $G(\pi)$ (that is, there is an occurrence of $\pi$ in $\sigma$) such that the regions of $G(\pi)$ which get shaded in the mesh corresponded to regions in the graph of $\sigma$ that contain no points of $G(\sigma)$.
\end{defn}

\begin{ex}
The permutation $42135$ contained the pattern $213$ in five ways.  This same permutation contains the mesh pattern $(213,\{(0,3),(1,2),(1,3),(3,0)\})$ in only four ways, depicted below.  We use thick lines in this example to clarify how the four elements in the mesh appear in each picture.  For example, $(1,2) \in R$ describes the region whose horizontal coordinates are between the first and second symbols in the $213$-pattern, and whose vertical coordinates are between the first and third symbols in the $213$-pattern.
$$
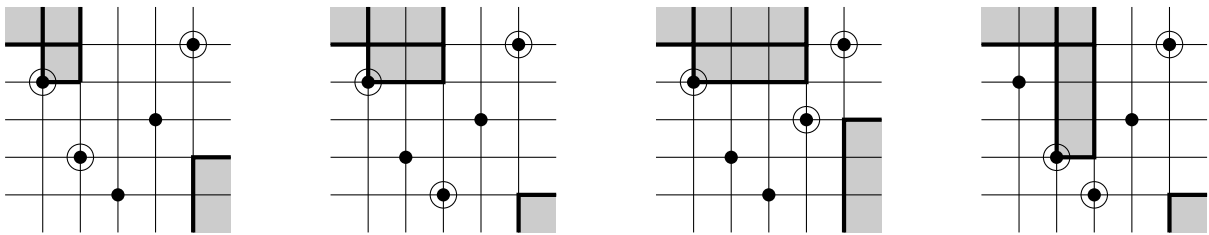
\begin{tikzpicture}[scale=.5]
\foreach \x in {(0,5),(1,4),(1,5),(5,0),(5,1)}{\fill[black!20] \x rectangle ++(1,1);}
\foreach \x in {1,2,3,4,5} {\draw (0,\x) -- (6,\x); \draw (\x,0) -- (\x,6);}
\foreach \x in {(1,4),(2,2),(3,1),(4,3),(5,5)} {\fill[black] \x circle (5pt);}
\foreach \x in {(1,4),(2,2),(5,5)}{\draw \x circle (10pt);}
\foreach \x in {1,2} {\draw[ultra thick] (\x,4) -- (\x,6);}
\draw[ultra thick] (1,4) -- (2,4); \draw[ultra thick] (0,5) -- (2,5);
\draw[ultra thick] (5,0) -- (5,2) -- (6,2);
\end{tikzpicture}\hspace{.5in}
\begin{tikzpicture}[scale=.5]
\fill[black!20] (1,4) rectangle (3,5);
\fill[black!20] (0,5) rectangle (3,6);
\fill[black!20] (5,0) rectangle (6,1);
\foreach \x in {1,2,3,4,5} {\draw (0,\x) -- (6,\x); \draw (\x,0) -- (\x,6);}
\foreach \x in {(1,4),(2,2),(3,1),(4,3),(5,5)} {\fill[black] \x circle (5pt);}
\foreach \x in {(1,4),(3,1),(5,5)}{\draw \x circle (10pt);}
\draw[ultra thick] (0,5) -- (1,5) -- (1,6);
\draw[ultra thick] (1,5) -- (3,5) -- (3,6);
\draw[ultra thick] (1,5) -- (1,4) -- (3,4) -- (3,5);
\draw[ultra thick] (5,0) -- (5,1) -- (6,1);
\end{tikzpicture}\hspace{.5in}
\begin{tikzpicture}[scale=.5]
\fill[black!20] (1,4) rectangle (4,5);
\fill[black!20] (0,5) rectangle (4,6);
\fill[black!20] (5,0) rectangle (6,3);
\foreach \x in {1,2,3,4,5} {\draw (0,\x) -- (6,\x); \draw (\x,0) -- (\x,6);}
\foreach \x in {(1,4),(2,2),(3,1),(4,3),(5,5)} {\fill[black] \x circle (5pt);}
\foreach \x in {(1,4),(4,3),(5,5)}{\draw \x circle (10pt);}
\draw[ultra thick] (0,5) -- (1,5) -- (1,6);
\draw[ultra thick] (1,5) -- (4,5) -- (4,6);
\draw[ultra thick] (1,5) -- (1,4) -- (4,4) -- (4,5);
\draw[ultra thick] (5,0) -- (5,3) -- (6,3);
\end{tikzpicture}\hspace{.5in}
\begin{tikzpicture}[scale=.5]
\fill[black!20] (0,5) rectangle (2,6);
\fill[black!20] (2,2) rectangle (3,6);
\fill[black!20] (5,0) rectangle (6,1);
\foreach \x in {1,2,3,4,5} {\draw (0,\x) -- (6,\x); \draw (\x,0) -- (\x,6);}
\foreach \x in {(1,4),(2,2),(3,1),(4,3),(5,5)} {\fill[black] \x circle (5pt);}
\foreach \x in {(2,2), (3,1), (5,5)}{\draw \x circle (10pt);}
\draw[ultra thick] (0,5) -- (2,5) -- (2,6);
\draw[ultra thick] (2,5) -- (3,5) -- (3,6);
\draw[ultra thick] (2,5) -- (2,2) -- (3,2) -- (3,5);
\draw[ultra thick] (5,0) -- (5,1) -- (6,1);
\end{tikzpicture}$$
The remaining occurrence of $213$ in $42135$ does not obey the restrictions of the mesh because the shaded region includes an element of $G(42135)$, as shown in the following diagram.
$$
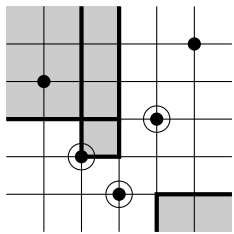
\begin{tikzpicture}[scale=.5]
\fill[black!20] (0,3) rectangle (2,6);
\fill[black!20] (2,2) rectangle (3,6);
\fill[black!20] (4,0) rectangle (6,1);
\foreach \x in {1,2,3,4,5} {\draw (0,\x) -- (6,\x); \draw (\x,0) -- (\x,6);}
\foreach \x in {(1,4),(2,2),(3,1),(4,3),(5,5)} {\fill[black] \x circle (5pt);}
\foreach \x in {(2,2), (3,1), (4,3)}{\draw \x circle (10pt);}
\draw[ultra thick] (0,3) -- (2,3) -- (2,6);
\draw[ultra thick] (2,3) -- (3,3) -- (3,6);
\draw[ultra thick] (2,3) -- (2,2) -- (3,2) -- (3,3);
\draw[ultra thick] (4,0) -- (4,1) -- (6,1);
\end{tikzpicture}$$
\end{ex}

In this paper, we are concerned with which permutations contain or avoid a given pattern.  To this end, we define two sets.

\begin{defn}
For a (classical or mesh) pattern $p$, let $\av(p)$ be the set of permutations that avoid $p$, and let $\cont(p)$ be the set of permutations that contain $p$.  Similarly, if $P$ is any collection of patterns, then
\begin{eqnarray*}
\av(P) &=& \bigcap_{p \in P} \av(p), \text{ and}\\
\cont(P) &=& \bigcup_{p \in P} \cont(p).
\end{eqnarray*}
\end{defn}

The following fact is obvious for any set $P$ of patterns:
\begin{equation}\label{eqn:av and cont}
\bigcup_n \mf{S}_n = \av(P) \sqcup \cont(P).
\end{equation}

\section{Main results}

In a personal communication, Kitaev asked which mesh patterns could be equivalently described by classical patterns.  That is, we want to know for which mesh patterns $(\pi,R)$ there exists a set $S(\pi,R)$ of classical patterns so that
$$\av((\pi,R)) = \av(S(\pi,R)).$$
By equation~\eqref{eqn:av and cont}, this property can be equivalently stated as
$$\cont((\pi,R)) = \cont(S(\pi,R)).$$

Throughout this paper, we will assume that the set $S(\pi,R)$, when it exists, is of minimal cardinality.  More precisely, if $\sigma$ is contained in $\tau$, then $\av(\{\sigma,\tau\}) = \av(\{\sigma\})$.  Thus it suffices to consider just $\{\sigma\}$ when discussing avoidance and containment of $\{\sigma,\tau\}$.

The possibility of different types of patterns characterizing the same sets of permutations has arisen before, and was the subject of a recent paper comparing barred and vincular avoidance \cite{coincidental}.  We will expand upon the language of that paper here.

\begin{defn}
Suppose that $P$ and $P'$ are two sets of permutation patterns.  If $\av(P) = \av(P')$, then $P$ and $P'$ are \emph{coincident}.  If $P = \{p\}$, then we may say that $p$ and $P'$ are \emph{coincident}.
\end{defn}

Recall that $P$ and $P'$ are Wilf-equivalent if $|\av(P) \cap \mf{S}_n| = |\av(P') \cap \mf{S}_n|$ for all $n$.  Thus coincidence is stronger than Wilf-equivalence because coincidence requires that the sets $\av(P)$ and $\av(P')$ themselves coincide.

Observe that the complement to Kitaev's question is entirely straightforward to answer.

\begin{lem}
Every classical permutation $\pi$ is coincident to the mesh pattern $(\pi,\emptyset)$.
\end{lem}

In order to answer Kitaev's question, we will need to introduce an additional piece of terminology.

\begin{defn}\label{defn:enclosed}
Let $(\pi,R)$ be a mesh pattern.  An \emph{enclosed diagonal} in $(\pi,R)$ is a triple $((i,j),\varepsilon, h)$ where $\varepsilon \in \{-1,1\}$, $h \ge 1$, and
\begin{itemize}
\item $\{(i+x,j+x\varepsilon) : 1 \le x < h\} \subseteq G(\pi)$,
\item $(i+x,j+x\varepsilon) \not\in G(\pi)$ for $x \in \{0,h\}$, and
\item $\{(i+x,j+x\varepsilon) : 0 \le x < h\} \subseteq R$.
\end{itemize}
\end{defn}

In the graph of a mesh pattern, the enclosed diagonals have one of the forms depicted in Figure~\ref{fig:enclosed}.  The terminology of Definition~\ref{defn:enclosed} refers to the fact that the diagonal of consecutive elements of $G(\pi)$ is entirely enclosed by elements of the mesh.

\begin{figure}[H]
$$
\begin{tikzpicture}[scale=.5]
\foreach \x in {0,1,3,4} {\fill[black!20] (\x,\x) rectangle ++(1,1); \draw(\x,\x) rectangle ++(1,1);}
\foreach \x in {1,2,3,4} {\fill[black] (\x,\x) circle (5pt);}
\foreach \x in {2.25,2.5,2.75} {\fill[black] (\x,\x) circle (1.5pt);}
\end{tikzpicture}
\hspace{1in}
\begin{tikzpicture}[scale=.5]
\foreach \x in {0,1,3,4} {\fill[black!20] (\x,-\x) rectangle ++(1,-1); \draw(\x,-\x) rectangle ++(1,-1);}
\foreach \x in {1,2,3,4} {\fill[black] (\x,-\x) circle (5pt);}
\foreach \x in {2.25,2.5,2.75} {\fill[black] (\x,-\x) circle (1.5pt);}
\end{tikzpicture}$$
\caption{Enclosed diagonals in a mesh pattern $(\pi,R)$.  A point is in $G(\pi)$ if and only if it is marked $\bullet$.}\label{fig:enclosed}
\end{figure}
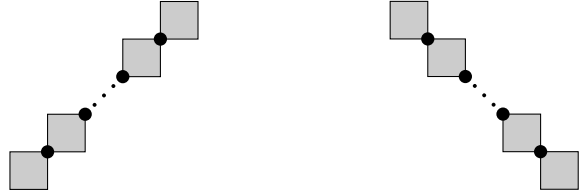


\begin{ex}\label{ex:enclosed diagonal}
The mesh pattern $(231, \{(1,1),(2,0),(3,1)\})$ depicted by
$$
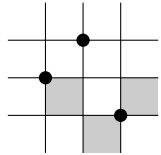
\begin{tikzpicture}[scale=.5]
\foreach \x in {(1,1),(2,0),(3,1)} {\fill[black!20] \x rectangle ++(1,1);}
\foreach \x in {1,2,3} {\draw (\x,0) -- (\x,4); \draw (0,\x) -- (4,\x);}
\foreach \x in {(1,2),(2,3),(3,1)} {\fill[black] \x circle (5pt);}
\end{tikzpicture}$$
has exactly one enclosed diagonal.  It is $\big((2,0), 1, 2\big)$.
\end{ex}

Note that when $h=1$ in Definition~\ref{defn:enclosed}, the value of $\varepsilon$ is irrelevant.  Also, in this case, the enclosed diagonal consists of a single square of the mesh.  This means that there is an $(i,j) \in R$ with $\{\pi(i),\pi(i+1)\} \cap \{j,j+1\} = \emptyset$.  In other words, the shaded square with lower-left corner $(i,j)$ must look like the square drawn below, where again a point is in $G(\pi)$ if and only if it is marked $\bullet$.
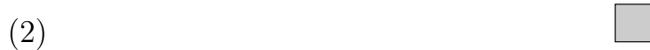
\begin{equation}\label{eqn:empty corners}
\begin{tikzpicture}[scale=.5]
\fill[black!20] (0,0) rectangle (1,1);
\draw (0,0) rectangle (1,1);
\end{tikzpicture}
\end{equation}
To put this another way, for a square $(i,j) \in R$ to not, itself, be an enclosed diagonal, it must have one of the following forms.
$$
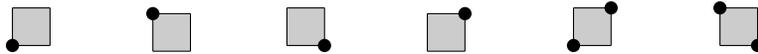
\begin{tikzpicture}[scale=.5]
\fill[black!20] (0,0) rectangle (1,1);
\foreach \x in {0,1} {\draw (0,\x) -- (1,\x); \draw (\x,0) -- (\x,1);}
\fill[black] (0,0) circle (5pt);
\end{tikzpicture}
\hspace{.5in}
\begin{tikzpicture}[scale=.5]
\fill[black!20] (0,0) rectangle (1,1);
\foreach \x in {0,1} {\draw (0,\x) -- (1,\x); \draw (\x,0) -- (\x,1);}
\fill[black] (0,1) circle (5pt);
\end{tikzpicture}
\hspace{.5in}
\begin{tikzpicture}[scale=.5]
\fill[black!20] (0,0) rectangle (1,1);
\foreach \x in {0,1} {\draw (0,\x) -- (1,\x); \draw (\x,0) -- (\x,1);}
\fill[black] (1,0) circle (5pt);
\end{tikzpicture}
\hspace{.5in}
\begin{tikzpicture}[scale=.5]
\fill[black!20] (0,0) rectangle (1,1);
\foreach \x in {0,1} {\draw (0,\x) -- (1,\x); \draw (\x,0) -- (\x,1);}
\fill[black] (1,1) circle (5pt);
\end{tikzpicture}
\hspace{.5in}
\begin{tikzpicture}[scale=.5]
\fill[black!20] (0,0) rectangle (1,1);
\foreach \x in {0,1} {\draw (0,\x) -- (1,\x); \draw (\x,0) -- (\x,1);}
\fill[black] (0,0) circle (5pt);
\fill[black] (1,1) circle (5pt);
\end{tikzpicture}
\hspace{.5in}
\begin{tikzpicture}[scale=.5]
\fill[black!20] (0,0) rectangle (1,1);
\foreach \x in {0,1} {\draw (0,\x) -- (1,\x); \draw (\x,0) -- (\x,1);}
\fill[black] (0,1) circle (5pt);
\fill[black] (1,0) circle (5pt);
\end{tikzpicture}
$$
These are the only possibilities because $\pi$ is a permutation, and so $G(\pi)$ contains exactly one element along each column and along each row.

We are now able to state the main result of this paper.

\begin{thm}\label{thm:main}
A mesh pattern $(\pi,R)$ is coincident to a set $S(\pi,R)$ of classical patterns if and only if $(\pi,R)$ has no enclosed diagonals.  Moreover, if $(\pi,R)$ satisfies this requirement, then in fact $S(\pi,R) = \{\pi\}$, and so $(\pi,R)$ is coincident to $\pi$.
\end{thm}

\begin{ex}\
\begin{itemize}
\item The mesh pattern $(231, \{(1,1)\})$ depicted by
$$
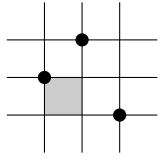
\begin{tikzpicture}[scale=.5]
\fill[black!20] (1,1) rectangle (2,2);
\foreach \x in {1,2,3} {\draw (\x,0) -- (\x,4); \draw (0,\x) -- (4,\x);}
\foreach \x in {(1,2),(2,3),(3,1)} {\fill[black] \x circle (5pt);}
\end{tikzpicture}$$
is coincident to the permutation $231$.
\item The mesh pattern $(231, \{(1,1),(3,2)\})$ depicted by
$$
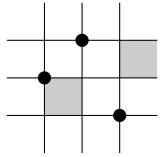
\begin{tikzpicture}[scale=.5]
\foreach \x in {(1,1),(3,2)} {\fill[black!20] \x rectangle ++(1,1);}
\foreach \x in {1,2,3} {\draw (\x,0) -- (\x,4); \draw (0,\x) -- (4,\x);}
\foreach \x in {(1,2),(2,3),(3,1)} {\fill[black] \x circle (5pt);}
\end{tikzpicture}$$
is not coincident to any set of classical patterns because $\big((3,2), 1, 1\big)$ is an enclosed diagonal.
\item The mesh pattern $(231, \{(1,1),(2,0),(3,1)\})$ of Example~\ref{ex:enclosed diagonal} is not coincident to any set of classical patterns because of its enclosed diagonal.
\end{itemize}
\end{ex}

In light of the last statement of Theorem~\ref{thm:main}, we introduce the following terminology.

\begin{defn}
If a mesh pattern $(\pi,R)$ is coincident to the classical pattern $\pi$, then the mesh $R$ is \emph{superfluous}.
\end{defn}

Thus Theorem~\ref{thm:main} could be rephrased as follows.

{\renewcommand{\thethm}{\ref{thm:main}$^{\boldsymbol{\prime}}$}
\begin{thm}
A mesh pattern has a superfluous mesh if and only if it has no enclosed diagonals.
\end{thm}
\addtocounter{thm}{-1}
}

Related to one direction of Theorem~\ref{thm:main} is the Shading Lemma of \cite{hilmarsson jonsdottir sigurdardottir vidarsdottir ulfarsson}, discovered independently.

One application of Theorem~\ref{thm:main} recovers a result about so-called ``boxed'' mesh patterns from \cite{avgustinovich kitaev valyuzhenich}.

\begin{cor}[{{\cite[Proposition 1]{avgustinovich kitaev valyuzhenich}}}]
The only permutations of $k$ letters for which $[1,k-1]\times[1,k-1]$ is a superfluous mesh are
$1$, $12$, $21$, $132$, $213$, $231$, and $312$.
\end{cor}

\begin{proof}
Suppose that $\pi \in \mf{S}_k$ is a permutation for which $[1,k-1]\times[1,k-1]$ is a superfluous mesh.  Given \eqref{eqn:empty corners}, and the fact that $\pi$ is a bijection on $[k]$, we can conclude that $k \le 4$.

If $k=4$, then it is easy to see that there will always be at least one square in the mesh that looks like \eqref{eqn:empty corners}, such as the darkly shaded square indicated below, yielding an enclosed diagonal.
$$
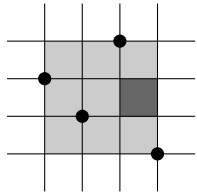
\begin{tikzpicture}[scale=.5]
\fill[black!20] (1,1) rectangle (4,4);
\fill[black!60] (3,2) rectangle (4,3);
\foreach \x in {1,2,3,4} {\draw (0,\x) -- (5,\x); \draw (\x,0) -- (\x,5);}
\foreach \x in {(1,3),(2,2),(3,4),(4,1)} {\fill[black] \x circle (5pt);}
\end{tikzpicture}$$
The permutations with $k\le 3$ can each be checked by hand, and exactly two of these have an enclosed diagonal, marked below with darkly shaded squares.
$$
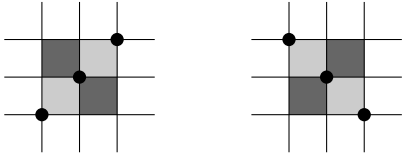
\begin{tikzpicture}[scale=.5]
\fill[black!20] (1,1) rectangle (3,3);
\foreach \x in {(1,2),(2,1)} {\fill[black!60] \x rectangle ++(1,1);}
\foreach \x in {1,2,3} {\draw (0,\x) -- (4,\x); \draw (\x,0) -- (\x,4); \fill[black] (\x,\x) circle (5pt);}
\end{tikzpicture}
\hspace{.5in}
\begin{tikzpicture}[scale=.5]
\fill[black!20] (1,1) rectangle (3,3);
\foreach \x in {(1,1),(2,2)} {\fill[black!60] \x rectangle ++(1,1);}
\foreach \x in {1,2,3} {\draw (0,\x) -- (4,\x); \draw (\x,0) -- (\x,4); \fill[black] (\x,4-\x) circle (5pt);}
\end{tikzpicture}$$
\end{proof}

\section{Proof of Theorem~\ref{thm:main}}\label{section:proof}

The main result of this paper can be proved in three steps, which we break into propositions below.  The first of these will prove the final sentence in the statement of the theorem.

\begin{prop}
If a mesh pattern $(\pi,R)$ is coincident to a set $S(\pi,R)$ of classical patterns, then $S(\pi,R) = \{\pi\}$.
\end{prop}

\begin{proof}
Because $\pi$ is the unique element of $\cont((\pi,R))$ using the fewest letters, we must have $\pi \in S(\pi,R)$.

Suppose there is some other element $\sigma \in S(\pi,R) \setminus \{\pi\}$.  By the minimality of $|S(\pi,R)|$, we can assume that $\sigma$ does not contain a $\pi$-pattern.  However, because $\sigma \in \cont(S(\pi,R)) = \cont((\pi,R))$, we must have that $\pi$ is contained in $\sigma$ after all, obtaining a contradiction.
\end{proof}

We now address the biconditional statement of Theorem~\ref{thm:main}.  The first direction of this is below.

\begin{prop}
If a mesh pattern $(\pi,R)$ has an enclosed diagonal, then its mesh $R$ is not superfluous.
\end{prop}

\begin{proof}
Suppose that the mesh pattern $(\pi,R)$ has an enclosed diagonal $((i,j), \varepsilon, h)$.  Without loss of generality, let $\varepsilon = 1$.  Then
\begin{eqnarray*}
&\{(i+x,j+x) : 1 \le x < h\} \subseteq G(\pi),&\\
&(i+x,j+x) \not\in G(\pi) \text{ for } x \in \{0,h\}, \text{ and}&\\
&\{(i+x,j+x) : 0 \le x < h\} \subseteq R.&\\
\end{eqnarray*}
Consider the permutation $\sigma$ that is order isomorphic to
$$\pi(1) \ \cdots \ \pi(i) \ (j + \textstyle{\frac{1}{2}}) \ \pi(i+1) \ \pi(i+2) \ \cdots,$$
formed by inserting $j+\frac{1}{2}$ between the $i$th and $(i+1)$st symbols in $\pi$.

By construction, $\sigma \in \cont(\pi)$.  In fact, $\sigma$ contains $h+1$ occurrences of $\pi$ --- obtained by choosing which of $\{\sigma(i+1), \sigma(i+2),\ldots, \sigma(i+h+1)\}$ should play the roles of $\{\pi(i+1),\pi(i+2),\ldots,\pi(i+h)\}$ in the pattern.  However, none of these occurrences of $\pi$ can be drawn so as to avoid all of the $h+1$ shaded boxes in the enclosed diagonal of $(\pi,R)$.  Thus $\sigma \not\in \cont((\pi,R))$.

Hence $\cont((\pi,R)) \neq \cont(\pi)$, and so $(\pi,R)$ and $\pi$ are not coincident.  In other words, the mesh $R$ is not superfluous.
\end{proof}

There is now one piece remaining in the proof of Theorem~\ref{thm:main}.

\begin{prop}
If a mesh pattern $(\pi,R)$ does not have a superfluous mesh, then it has an enclosed diagonal.
\end{prop}

\begin{proof}
Suppose that the mesh pattern $(\pi,R)$ is not coincident to $\pi$.  Certainly $\cont((\pi,R))\subseteq \cont(\pi)$, so consider some $\sigma \in \cont(\pi) \setminus \cont((\pi,R))$.  By construction, this $\sigma$ has at least one occurrence of the pattern $\pi$, and no occurrence of $\pi$ in $\sigma$ avoids all of the shaded regions from the mesh $R$.

Given an occurrence $\langle \pi \rangle$ of $\pi$ in $\sigma$, let $(\langle \pi \rangle, R)_{\sigma}$ denote the number of elements of $G(\sigma)$ that land in the shaded regions from the mesh $R$ relative to this $\langle \pi \rangle$.

Now choose an occurrence $\langle \pi \rangle$ for which $(\langle \pi \rangle, R)_{\sigma}$ is minimal, and let $(i,j) \in R$ correspond to a shaded region that contains some $(z,\sigma(z)) \in G(\sigma)$.  Note that because $(z,\sigma(z))$ lies in a shaded region of this occurrence of the mesh pattern, we necessarily have that $\sigma(z)$ itself is not part of $\langle \pi \rangle$.

If this $(i,j) \in R$ is itself in an enclosed diagonal in $(\pi,R)$, then we are done.  

Otherwise, without loss of generality, there exists $h \ge 1$ such that 
\begin{eqnarray}
\nonumber &\{(i+x,j+x) : 0 \le x < h\} \subseteq R,&\\
\nonumber &(i+h,j+h) \not\in R, \text{ and}&\\
\label{eqn:consecutive values} &\pi(i+x) = j+x \text{ for all } 1 \le x \le h.&
\end{eqnarray}

For all $1 \le x \le h$, let $\sigma(l_x)$ represent $\pi(i+x)$ in the occurrence $\langle \pi \rangle$, and note from equation~\eqref{eqn:consecutive values} that $\pi(i+1),\pi(i+2),\ldots,\pi(i+h)$ are consecutive increasing values in the pattern.  This means that $\sigma(l_1),\sigma(l_2),\ldots,\sigma(l_h)$ are increasing values in $\sigma$, and that no values in the occurrence $\langle \pi \rangle$ are strictly between $\sigma(l_x)$ and $\sigma(l_{x+1})$ for any $1 \le x < h$.

Now define $\langle \pi \rangle'$ from $\langle \pi \rangle$ by replacing $\{\sigma(l_1),\ldots,\sigma(l_h)\}$ with $\{\sigma(z),\sigma(l_1),\ldots,\sigma(l_{h-1})\}$, and not changing any other elements of the pattern.  By construction, this $\langle \pi \rangle'$ is another occurrence of $\pi$ in $\sigma$.  Also, in the diagram $G(\sigma)$, the shaded regions from the mesh change as follows, where a point in $G(\sigma)$ is filled in if and only if it is part of the pattern occurrence, and a region is shaded if and only if it is shaded by the mesh $R$. 
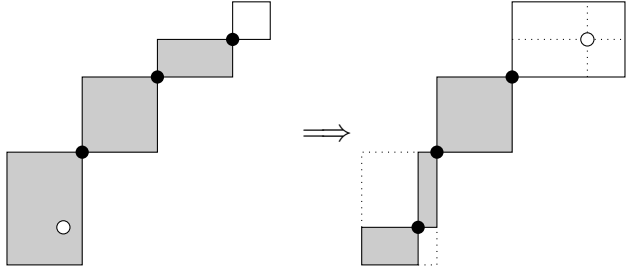
$$
\begin{minipage}{1.5in}
\begin{tikzpicture}[scale=.5]
\fill[black!20] (0,0) rectangle (2,3);
\fill[black!20] (2,3) rectangle (4,5);
\fill[black!20] (4,5) rectangle (6,6);
\draw (6,6) rectangle (7,7);
\draw (4,5) rectangle (6,6);
\draw (2,3) rectangle (4,5);
\draw (0,0) rectangle (2,3);
\foreach \x in {(2,3),(4,5),(6,6)} {\fill[black] \x circle (5pt);}
\foreach \x in {(1.5,1)} {\fill[white] \x circle (5pt); \draw \x circle (5pt);}
\end{tikzpicture}
\end{minipage}
\Longrightarrow
\begin{minipage}{1.5in}
\begin{tikzpicture}[scale=.5]
\fill[black!20] (0,0) rectangle (1.5,1);
\fill[black!20] (1.5,1) rectangle (2,3);
\fill[black!20] (2,3) rectangle (4,5);
\draw (4,5) rectangle (7,7);
\draw (2,3) rectangle (4,5);
\draw (1.5,1) rectangle (2,3);
\draw (0,0) rectangle (1.5,1);
\foreach \x in {(1.5,1),(2,3),(4,5)} {\fill[black] \x circle (5pt);}
\draw[dotted] (1.5,0) -- (2,0) -- (2,1); \draw[dotted] (0,1) -- (0,3) -- (1.5,3);
\draw[dotted] (4,6) -- (7,6);\draw[dotted] (6,5) -- (6,7);
\foreach \x in {(6,6)} {\fill[white] \x circle (5pt); \draw \x circle (5pt);}
\end{tikzpicture}
\end{minipage}
$$

This yields
$$(\langle \pi \rangle', R)_{\sigma} < (\langle \pi \rangle, R)_{\sigma},$$
which contradicts the minimality assumption of $(\langle \pi \rangle, R)_{\sigma}$.  Thus $(i,j)$ must have been part of an enclosed diagonal.
\end{proof}

\section{Extremal enumeration}

Given a permutation $\pi$, it is not difficult to enumerate the mesh patterns $(\pi,R)$ for which the mesh $R$ is superfluous.  This is because, by Theorem~\ref{thm:main}, we must simply choose an $R$ that has no enclosed diagonals.

\begin{defn}
Given a permutation $\pi$, let $\supmesh(\pi)$ be the number of meshes that are superfluous for $\pi$.
\end{defn}

To calculate \supmesh for a given permutation, we can use the inclusion-exclusion principle.

From there, it is also straightforward to enumerate the meshes that are not superfluous for some $\pi \in \mf{S}_k$, because these are all the remaining subsets of $[0,k]\times[0,k]$: $2^{(k+1)^2} - \supmesh(\pi)$.

\begin{ex}\label{ex:123 superfluous} 
Consider $123 \in \mf{S}_3$.  Because any square of the form depicted in \eqref{eqn:empty corners} would itself constitute an enclosed diagonal, each $(i,j)$ in a superfluous mesh for $123$ must satisfy $|i-j| \le 1$.  The squares satisfying this requirement are shaded below.
$$
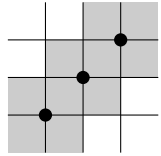
\begin{tikzpicture}[scale=.5]
\foreach \x in {(0,0), (0,1), (1,0), (1,1), (1,2), (2,1), (2,2), (2,3), (3,2), (3,3)} {\fill[black!20] \x rectangle ++(1,1);}
\foreach \x in {1,2,3} {\draw (0,\x) -- (4,\x); \draw (\x,0) -- (\x,4); \fill[black] (\x,\x) circle (5pt);}
\end{tikzpicture}$$
With this restriction in place, the only other enclosed diagonals to worry about are the following four possibilities.
$$
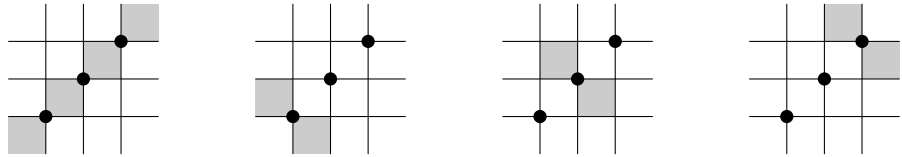
\begin{tikzpicture}[scale=.5]
\foreach \x in {0,1,2,3} {\fill[black!20] (\x,\x) rectangle ++(1,1);}
\foreach \x in {1,2,3} {\draw (0,\x) -- (4,\x); \draw (\x,0) -- (\x,4); \fill[black] (\x,\x) circle (5pt);}
\end{tikzpicture}
\hspace{.5in}
\begin{tikzpicture}[scale=.5]
\foreach \x in {(1,0),(0,1)} {\fill[black!20] \x rectangle ++(1,1);}
\foreach \x in {1,2,3} {\draw (0,\x) -- (4,\x); \draw (\x,0) -- (\x,4); \fill[black] (\x,\x) circle (5pt);}
\end{tikzpicture}
\hspace{.5in}
\begin{tikzpicture}[scale=.5]
\foreach \x in {(2,1),(1,2)} {\fill[black!20] \x rectangle ++(1,1);}
\foreach \x in {1,2,3} {\draw (0,\x) -- (4,\x); \draw (\x,0) -- (\x,4); \fill[black] (\x,\x) circle (5pt);}
\end{tikzpicture}
\hspace{.5in}
\begin{tikzpicture}[scale=.5]
\foreach \x in {(3,2),(2,3)} {\fill[black!20] \x rectangle ++(1,1);}
\foreach \x in {1,2,3} {\draw (0,\x) -- (4,\x); \draw (\x,0) -- (\x,4); \fill[black] (\x,\x) circle (5pt);}
\end{tikzpicture}
$$
We now use inclusion-exclusion to count the superfluous meshes for $123$:
\begin{eqnarray*}
\supmesh(123) &=& 2^{10} - (2^6 + 2^8 + 2^8 + 2^8) + (2^4 + 2^4 + 2^4 + 2^6 + 2^6 + 2^6)\\
&& \phantom{2^{10}} - (2^2 + 2^2 + 2^2 + 2^4) + 2^0\\
&=& 405.
\end{eqnarray*}
\end{ex}

\begin{ex}\label{ex:132 superfluous} 
Consider $132 \in \mf{S}_3$.  As in the previous example, the only possible squares that can be shaded by a superfluous mesh for $132$ are indicated below.
$$
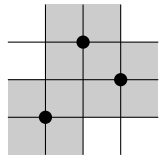
\begin{tikzpicture}[scale=.5]
\foreach \x in {(0,0), (0,1), (1,0), (1,1), (1,2), (1,3), (2,1), (2,2), (2,3), (3,1), (3,2)} {\fill[black!20] \x rectangle ++(1,1);}
\foreach \x in {1,2,3} {\draw (0,\x) -- (4,\x); \draw (\x,0) -- (\x,4);}
\foreach \x in {(1,1), (2,3), (3,2)} {\fill[black] \x circle (5pt);}
\end{tikzpicture}$$
There are five other enclosed diagonals that we must also avoid.
$$
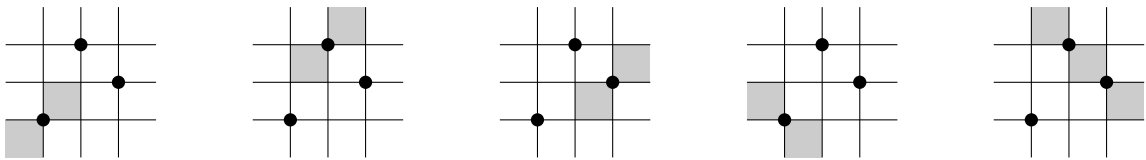
\begin{tikzpicture}[scale=.5]
\foreach \x in {(0,0), (1,1)} {\fill[black!20] \x rectangle ++(1,1);}
\foreach \x in {1,2,3} {\draw (0,\x) -- (4,\x); \draw (\x,0) -- (\x,4);}
\foreach \x in {(1,1), (2,3), (3,2)} {\fill[black] \x circle (5pt);}
\end{tikzpicture}
\hspace{.5in}
\begin{tikzpicture}[scale=.5]
\foreach \x in {(1,2), (2,3)} {\fill[black!20] \x rectangle ++(1,1);}
\foreach \x in {1,2,3} {\draw (0,\x) -- (4,\x); \draw (\x,0) -- (\x,4);}
\foreach \x in {(1,1), (2,3), (3,2)} {\fill[black] \x circle (5pt);}
\end{tikzpicture}
\hspace{.5in}
\begin{tikzpicture}[scale=.5]
\foreach \x in {(2,1), (3,2)} {\fill[black!20] \x rectangle ++(1,1);}
\foreach \x in {1,2,3} {\draw (0,\x) -- (4,\x); \draw (\x,0) -- (\x,4);}
\foreach \x in {(1,1), (2,3), (3,2)} {\fill[black] \x circle (5pt);}
\end{tikzpicture}
\hspace{.5in}
\begin{tikzpicture}[scale=.5]
\foreach \x in {(0,1), (1,0)} {\fill[black!20] \x rectangle ++(1,1);}
\foreach \x in {1,2,3} {\draw (0,\x) -- (4,\x); \draw (\x,0) -- (\x,4);}
\foreach \x in {(1,1), (2,3), (3,2)} {\fill[black] \x circle (5pt);}
\end{tikzpicture}
\hspace{.5in}
\begin{tikzpicture}[scale=.5]
\foreach \x in {(1,3), (2,2), (3,1)} {\fill[black!20] \x rectangle ++(1,1);}
\foreach \x in {1,2,3} {\draw (0,\x) -- (4,\x); \draw (\x,0) -- (\x,4);}
\foreach \x in {(1,1), (2,3), (3,2)} {\fill[black] \x circle (5pt);}
\end{tikzpicture}
$$
Thus
\begin{eqnarray*}
\supmesh(132) &=& 2^{11} - (2^9 + 2^9 + 2^9 + 2^9 + 2^8)\\
&& \phantom{2^{11}} + (2^7 + 2^7 + 2^7 + 2^6 + 2^7 + 2^7 + 2^6 + 2^7 + 2^6 + 2^6) \\
&& \phantom{2^{11}} - (2^5 + 2^5 + 2^4 + 2^5 + 2^4 + 2^4 + 2^5 + 2^4 + 2^4 + 2^4)\\
&& \phantom{2^{11}} + (2^3 + 2^2 + 2^2 + 2^2 + 2^2) - 2^0\\
&=& 576.
\end{eqnarray*}
\end{ex}

From these examples, it is clear how to find those permutations that have the fewest and the most superfluous meshes.  These extremes can be achieved by looking at how two different configurations
$$
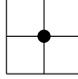
\begin{tikzpicture}[scale=.5]
\foreach \x in {0,1,2} {\draw (\x,0) -- (\x,2);\draw (0,\x) -- (2,\x);}
\fill[black] (1,1) circle (5pt);
\end{tikzpicture}$$
in $G(\pi)$ can overlap; in other words, by looking at when $|\pi(i) - \pi(i+1)| = 1$.

\begin{cor}\label{cor:extremal}
Suppose $\pi \in \mf{S}_k$.
\begin{enumerate}\renewcommand{\labelenumi}{(\alph{enumi})}
\item $\supmesh(\pi)$ is minimized if and only if $|\pi(i) - \pi(i+1)| = 1$ for all $1 \le i < k$; that is, if and only if $\pi = 123\cdots k$ or $\pi = k \cdots 321$.
\item $\supmesh(\pi)$ is maximized if and only if $|\pi(i) - \pi(i+1)| \neq 1$ for all $1 \le i < k$; that is, there are maximally many ($4k$) enclosed diagonals to avoid.
\end{enumerate}
\end{cor}

\begin{proof}
This result relies on the fact that if $|\pi(i) - \pi(i+1)|$ equals $1$, then there will be fewer superfluous meshes possible than if $|\pi(i) - \pi(i+1)|$ does not equal $1$.

To see why this is the case, suppose that $\pi(i)$ is the $I$th in a sequence of increasing consecutive values in $\pi$, and the $D$th in a sequence of decreasing consecutive values in $\pi$, where certainly $\min\{I,D\} = 1$.  More precisely,
$$\pi(i-1) = \pi(i) - 1,\ \pi(i-2) = \pi(i) - 2,\ \ldots, \ \pi(i-I+1) = \pi(i) - I + 1,$$
and
$$\pi(i-1) = \pi(i) + 1,\ \pi(i-2) = \pi(i) + 2,\ \ldots, \ \pi(i-D+1) = \pi(i) + D-1.$$
Similarly, suppose that $\pi(i+1)$ is first in a sequence of $I'$ increasing consecutive values in $\pi$, and first in a sequence of $D'$ decreasing consecutive values in $\pi$, where again $\min\{I',D'\} = 1$.

If $\pi(i) - \pi(i+1) = 1$, then $\{(i,\pi(i)),(i+1,\pi(i+1))\} \in G(\pi)$ could belong to three possible enclosed diagonals: one with $I+1$ shaded squares, one with $I'+1$ shaded squares, and one with $D+D'+1$ shaded squares.  The case of $\pi(i+1) - \pi(i) = 1$ is analogous.  On the other hand, if $|\pi(i) - \pi(i+1)| \neq 1$, then there are four enclosed diagonals to avoid: one with $I+1$ shaded squares, one with $D+1$ shaded squares, one with $I'+1$ shaded squares, and one with $D'+1$ shaded squares.

Now, for each case, count the superfluous meshes involving the squares along these diagonals.  The difference between this number in the latter situation and this number in the former is
\begin{align*}
&\Big[2^{I+I'+D+D'+4} - (2^{I+I'+D+3} + 2^{I+I'+D'+3} + 2^{I+D+D'+3} + 2^{I'+D+D'+3})\\
& \hspace{.75in} + (2^{I+I'+2} + 2^{I+D+2} + 2^{I+D'+2} + 2^{I'+D+2} + 2^{I'+D'+2} + 2^{D+D'+2})\\
& \hspace{.75in} - (2^{I+1} + 2^{I'+1} + 2^{D+1} + 2^{D'+1}) + 2^0\Big]\\
& \hspace{.25in} - \Big[2^{I+I'+D+D'+3} - (2^{I+I'+2} + 2^{I+D+D'+2} + 2^{I'+D+D'+2}) + (2^{I+1} + 2^{I'+1} + 2^{D+D'+2}) - 2^0\Big]\\
& \hspace{.25in} = 2(2^{I+1} - 1)(2^{I'+1}-1)(2^{D}-1)(2^{D'}-1)\\
& \hspace{.25in} > 0.
 \end{align*}
\end{proof}

The permutations characterized in Corollary~\ref{cor:extremal}(b) can also be described as avoiding the so-called 1-box pattern, as discussed in \cite{kitaev remmel}.  Additionally, these permutations are enumerated by sequence A002464 of \cite{oeis}.

Applying the same counting techniques as in Examples~\ref{ex:123 superfluous} and~\ref{ex:132 superfluous}, we can enumerate the superfluous meshes for each extremal case described in Corollary~\ref{cor:extremal}.

\begin{cor}
Fix a positive integer $k$.  Then
$$\min_{\pi \in \mathfrak{S}_k} \ \supmesh (\pi) = \sum_{i=0}^{k+1} (-1)^i\left(\binom{k}{i-1} 2^{2k-2i+2} + \binom{k}{i} 2^{3k+1-2i}\right),$$
where $\binom{b}{a} = 0$ for $a<0$ or $b<a$, and
$$\max_{\pi \in \mathfrak{S}_k} \ \supmesh (\pi) = \sum_{i=0}^k (-1)^i \binom{2k}{i} 2^{4k-2i}.$$
\end{cor}

\begin{proof}
In the first case, where the permutation is either $123\cdots k$ or $k \cdots 321$, there will be $3k+1$ potential elements of a superfluous mesh.  These possible squares can be partitioned into $k+1$ enclosed diagonals that need to be avoided: one of length $k+1$, and $k$ of length $2$.  The following picture gives an example of this partition, with the arrows indicating the (hazardous) enclosed diagonals.
$$
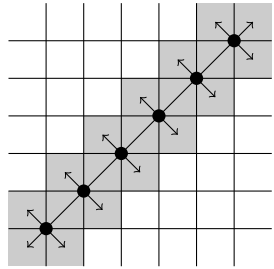
\begin{tikzpicture}[scale=.5]
\foreach \x in {(0,0), (0,1), (1,0), (1,1), (1,2), (2,1), (2,2), (2,3), (3,2), (3,3), (3,4), (4,3), (4,4), (4,5), (5,4), (5,5),(5,6),(6,5),(6,6)} {\fill[black!20] \x rectangle ++(1,1);}
\foreach \x in {1,2,3,4,5,6} {\draw (0,\x) -- (7,\x); \draw (\x,0) -- (\x,7);\fill[black] (\x,\x) circle (5pt);}
\draw[<->] (.5,.5) -- (6.5,6.5);
\foreach \x in {.5,1.5,2.5,3.5,4.5,5.5} {\draw[<->] (\x,\x+1) -- ++(1,-1);}
\end{tikzpicture}$$
The result follows from a calculation as in Examples~\ref{ex:123 superfluous} and~\ref{ex:132 superfluous}.

The only difference in the second case is that now there are $4k$ potential elements of a superfluous mesh, and these can be partitioned into $2k$ enclosed diagonals of length $2$ that must be avoided.  The following picture gives an example of this partition, and again the arrows indicate the (hazardous) enclosed diagonals.
$$
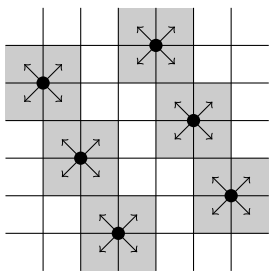
\begin{tikzpicture}[scale=.5]
\foreach \x in {(0,4),(0,5),(1,4),(1,5),(1,2),(1,3),(2,2),(2,3),(2,0),(2,1),(3,0),(3,1),(3,5),(3,6),(4,5),(4,6),(4,3),(4,4),(5,3),(5,4),(5,1),(5,2),(6,1),(6,2))} {\fill[black!20] \x rectangle ++(1,1);}
\foreach \x in {1,2,3,4,5,6} {\draw (0,\x) -- (7,\x); \draw (\x,0) -- (\x,7);}
\foreach \x in {.5,1.5,2.5}{\fill[black] (\x+.5,6-\x-\x) circle (5pt); \draw[<->] (\x,6.5-\x - \x) -- ++(1,-1); \draw[<->] (\x+1,6.5-\x-\x) -- ++(-1,-1);}
\foreach \x in {3.5,4.5,5.5}{\fill[black] (\x+.5,13-\x-\x) circle (5pt); \draw[<->] (\x,13.5-\x-\x) -- ++(1,-1); \draw[<->] (\x+1,13.5-\x-\x) -- ++(-1,-1);}
\end{tikzpicture}$$
\end{proof}

\section*{Acknowledgements}

I am grateful to Sergey Kitaev for thoughtful comments on a draft of this paper, as well as for suggesting the problem in the first place.  I also appreciate the input from an anonymous referee, particularly regarding terminology.


\begin{thebibliography}{99}

\bibitem{avgustinovich kitaev valyuzhenich} S.~Avgustinovich, S.~Kitaev, and A.~Valyuzhenich, Avoidance of boxed mesh patterns on permutations, \textit{Discrete Appl.~Math.} \textbf{161} (2013), 43--41.

\bibitem{babson steingrimsson} E.~Babson and E.~Steingr\'imsson, Generalized permutation patterns and a classification of the Mahonian statistics, \textit{S\'em.~Lothar.~Combin.}, B44b (2000), 18pp.

\bibitem{bousquet-melou claesson dukes kitaev} M.~Bousquet-M\'elou, A.~Claesson, M.~Dukes, and S.~Kitaev, $(2+2)$-free posets, ascent sequences and pattern avoiding permutations, \textit{J.~Combin.~Theory Ser.~A} \textbf{117} (2010), 884--909.

\bibitem{branden claesson} P.~Br\"and\'en and A.~Claesson, Mesh patterns and the expansion of permutation statistics as sums of permutation patterns, \textit{Electron.~J.~Combin.} \textbf{18(2)} (2011), P5.

\bibitem{hilmarsson jonsdottir sigurdardottir vidarsdottir ulfarsson} \'I.~Hilmarsson, I.~J\'onsd\'ottir, S.~Sigur\dh ard\'ottir, S.~Vi\dh arsd\'ottir, H.~\'Ulfarsson, Partial Wilf-classification of small mesh patterns, in preparation.

\bibitem{kitaev remmel} S.~Kitaev and J.~Remmel, The $1$-box pattern on permutations and words, in preparation.

\bibitem{simion schmidt} R.~Simion and F.~W.~Schmidt, Restricted permutations, \textit{European J.~Combin.} \textbf{6} (1985), 383--406.

\bibitem{oeis}  N.~J.~A.~Sloane, The on-line encyclopedia of integer sequences, published electronically at \hfill \phantom{*} {\tt http:/$\!\!$/www.oeis.org/}.

\bibitem{steingrimsson} E.~Steingr\'imsson, Generalized permutation patterns -- a short survey, in \textit{Permutation Patterns}, London Math.~Soc.~Lecture Note Ser., Cambridge University Press, Cambridge, 2010, pp.~137--152.

\bibitem{coincidental} B.~E.~Tenner, Coincidental pattern avoidance, to appear in \textit{J.~Comb.}

\bibitem{ulfarsson} H. \'Ulfarsson, A unification of permutation patterns related to Schubert varieties, to appear in {Pure Math. Appl.} \textbf{22} (2011), 273--296.

\bibitem{west-thesis} J.~West, Permutations with forbidden subsequences and stack-sortable permutations, Ph.D.~thesis, MIT, 1990.

\bibitem{woo yong} A.~Woo and A.~Yong, When is a Schubert variety Gorenstein?, \textit{Adv.~Math.} \textbf{207} (2006), 205--220. 

\end{thebibliography}
\end{document}